\documentclass[12pt]{amsart}

\usepackage{amssymb}

\usepackage{enumerate}

\usepackage{graphicx}

\makeatletter
\@namedef{subjclassname@2010}{%
  \textup{2010} Mathematics Subject Classification}
\makeatother

\newtheorem{thm}{Theorem}[section]
\newtheorem*{thm*}{Main Theorem}
\newtheorem{cor}[thm]{Corollary}

\newtheorem{prop}[thm]{Proposition}

\newcommand{\R}{\mathbb{R}}

\newcommand{\inv}{\mbox{inv}}
\newcommand{\interior}{\mbox{int }}
\newcommand{\ip}[2]{\langle #1, #2 \rangle}

\theoremstyle{definition}
\newtheorem{defin}[thm]{Definition}
\newtheorem{rem}[thm]{Remark}

\numberwithin{equation}{section}

\usepackage[colorinlistoftodos,prependcaption]{todonotes}
\usepackage{xargs}
\usepackage{xcolor}
\newcommandx{\rMaciej}[2][1=]{\todo[#1]{Maciej: #2}}
\newcommandx{\rNils}[2][1=]{\todo[linecolor=blue,backgroundcolor=blue!25,bordercolor=blue,#1]{Nils: #2}}

\begin{document}


\baselineskip=17pt



\title[gradient proper maps]{Connected components of the space of proper gradient vector fields}

\author[M. Starostka]{Maciej Starostka}
\address{Ruhr Universitat Bochum \& Gdansk University of Technology}
\email{maciejstarostka@gmail.com}

\date{}

\begin{abstract}
We show that there exist two proper gradient vector fields on $\R^n$ which are homotopic in the category of proper maps but not homotopic in the category of proper gradient maps.
\end{abstract}



\maketitle

\section{Introduction}
Two vector fields on a closed unit ball $B^n \subset \R^n$ which do not vanish on $S^{n-1}$ are homotopic if and only if they have the same degree. In 1990, A. Parusinski asked and answered the following question: can we get a better invariant if we restrict the class of vector fields to gradient vector fields? The answer is negative. Namely (Theorem 1 in \cite{Par}), two gradient vector fields on $B^n$, which do no vanish on $S^{n-1}$, are gradient homotopic if and only if they are homotopic, i.e. if and only if they have the same degree. \\
The aim of this note is to show that we get a different answer if instead of vector fields on a ball we consider proper vector fields on $\R^n$. Again, two proper vector fields are homotopic if and only if they have the same degree. However, there exist proper gradient vector fields having the same degree, which are not gradient homotopic. \\

\begin{rem} The motivation for considering this problem comes from the study of invariants of $3$ and $4$-dimensional manifolds. The Bauer-Furuta type invariant is a homotopy class of a map which is a compact perturbation of a Fredholm operator and extends to the one point compactification between Hilbert spaces. Such classes are in one-to-one correspondence with elements of some stable homotopy groups (see \cite{Bauer}). On the other hand, there are invariants defined by Manolescu via the Conley index theory (\cite{Manolescu}). Our work suggests, that although the invariant defined via the Conley index theory is stronger than the corresponding Bauer-Furuta type invariant, this should no longer be the case if one restricts the class of homotopies to the category of gradient maps.
\end{rem}

Author would like to thank Kazimierz G\c{e}ba, Marek Izydorek and Thomas Rot for useful comments on the topic.

\section{Main Theorem}
Denote by $B(R) \subset \R^n$ a closed ball of radius $R$ centered at the orgin. We identify vector fields on $\R^n$ with maps from $\R^n$ into itself. Therefore a proper vector field extends to a continuous map between one point compactifications $S^n := S^{\R^n}$. Let $\mathcal{V}$ be the space of proper vector fields on $\R^n$ with a topology induced from $C^0(S^n,S^n)$. Let $F$ and $G$ be proper. In particular, we can choose $R > 0$ such that $F^{-1}(0)$, $G^{-1}(0) \subset B(R)$. Then, $F$ and $G$ are homotopic as proper maps, i.e. lie in the same connected component of $\mathcal{V}$, if and only if they have the same degree:
\[
\deg(F,B(R)) = \deg(G,B(R)).
\]
See \cite{Rot18} for a more detailed discussion. Now we would like to ask a question, whether the degree characterises also connected components of proper gradient vector fields.

\begin{defin}
Denote by $\mathcal{V}^\nabla$ subspace of $\mathcal{V}$ of gradient vector fields. We say that $F$ and $G \in \mathcal{V}^\nabla$ are \textit{gradient homotopic} if they lie in the same connected component of $\mathcal{V}^\nabla$.
\end{defin}
Our aim is to show the following.
\begin{thm*}
If $n \geqslant 2$ then there exist proper gradient vector fields on $\R^n$ which are homotopic but not gradient homotopic.
\end{thm*}

\subsection{Local Morse cohomology}
Let $\eta$ be a local flow on $\R^n$. 
\begin{defin}
We say that a closed and bounded set $U$ is an \emph{isolating neighbourhood} for $\eta$ if $\inv(U,\eta) \subset \interior U$, where $\inv(U,\eta) = \{x \in U| \eta(\R,x)\}$ is the invariant subset.
\end{defin}
The following Proposition is well-known (see \cite{Rot} or \cite{JDE} for more general setting).
\begin{prop}\label{prop:localMorse}
Let $f$ be a Morse-Smale function and  let $U$ be an isolating neighbourhood for the gradient flow of $f$. Then the local Morse cohomology groups $H^*(f,U)$ are well-defined. Moreover, let $\{\eta_\lambda\}_{\lambda \in [0,1]}$ be a continuous family of flows such that $U$ is an isolating neighbourhood for $\eta_\lambda$ for every $\lambda$ and $\eta_0, \eta_1$ are gradient flows for the Morse-Smale functions $f$ and $g$, respectively. Then
\[
H^*(f,U) \simeq H^*(g,U)
\]
\end{prop}

We will now show that a ball of big enough radius  is an isolating neigbourhood for a family of flows generated by proper gradient vector fields (compare \cite[pp. 56-57]{Schwarz}).

\begin{prop}\label{prop:maxInvSet}
Let $\{f_\lambda\}_{\lambda \in [0,1]}$ be a family of functions on $\R^n$ such that the corresponding family $\{\nabla f_\lambda \}_{\lambda \in [0,1]}$ is a continuous family of proper vector fields. Denote by $\eta_\lambda$ the flow generated by $\nabla f_\lambda$. Then there exists $R > 0$ such that $B(R)$ is an isolating neighbourhood for every $\eta_\lambda$ and
\[
\inv(B(R'),\eta_\lambda) = \inv(B(R),\eta_\lambda)
\]
for every $\lambda$ and every $R' > R$.
\end{prop}

\begin{proof}
By the properness, we can find an $r_1$ such that 
\[
\bigcup_{\lambda \in [0,1]}(\nabla f_\lambda)^{-1}(B(1)) \subset B(r_1).
\]
Put 
\[
r_2 = \max\{|f_\lambda(x)| : x \in B(r_1), \quad \lambda \in [0,1]\}
\]
We will show that $R = 2(r_1 + r_2)$ satisfies the statement. Let $u:\R \to \R^n$ be a bounded trajectory of the flow $\eta_\lambda$ for some $\lambda \in [0,1]$. We have to show that $|u(0)| < R$. Suppose that $|u(0)| > r_1$. Since $\alpha$-limit of $u$ is contained in the interior of $B(r_1)$, we can choose  $t_0 < 0$ such that $|u(t)| \geqslant r_1$ for every $t \in [t_0,0]$ and $|u(t_0)| = r_1$. Let $x$ and $y$ be points in $\alpha$ and $\omega$ limit of $u$, respectively. We have
\[ 
|x - u(0)| \leqslant |x - u(t_0)| + |u(t_0) - u(0)| \leqslant 2r_1 + \int_{t_0}^0|\dot{u}(s)| \, ds
\]
On the other hand, 
\[
|\dot{u}(s)| \leqslant |\dot{u}(s)|^2 =  \ip{\nabla f_\lambda (u(s))}{\dot{u}(s)} = \frac{d}{ds}(f_\lambda \circ u)(s)
\]
for $s \in [t_0,0]$. Finally,
\[
\int_{t_0}^0|\dot{u}(s)| \, ds \leqslant \int_{t_0}^0 \frac{d}{ds}(f_\lambda \circ u)(s) \, ds \leqslant f_\lambda(y) - f_\lambda(x) \leqslant 2r_2
\]
\end{proof}

By the Propositions \ref{prop:localMorse} and \ref{prop:maxInvSet} we get the following corollary.

\begin{cor}\label{cor:Invariance}
Let $f$ and $g$ be two Morse-Smale functions, such that $\nabla f$, $\nabla g$ are gradient homotopic. Then for $R \gg 0$ the Morse cohomology groups $H^*(f,B(R))$ and $H^*(g,B(R))$ are isomorphic.
\end{cor}

\subsection{Proof of the Main Theorem}
Put
\[
f(x_1,\ldots,x_n) = x_1^2 + \ldots x_n^2, \quad g(x_1,\ldots,x_n) = -x_1^2 -x_2^2+\ldots + x_n^2
\]
Then $\deg (\nabla f, B(R)) = 1 = \deg (\nabla g, B(R))$ so $\nabla f$ and $\nabla g$ are homotopic in the category of proper maps. However, $H^q(f,B(R))$ is non-zero only if $q = 0$ and $H^q(g,B(R))$ is non-zero only if $q = 2$ so by Corollary \ref{cor:Invariance} they are not gradient homotopic.
\section{Final remarks}
\subsection{Infinite dimensional case}
Note first that instead of using Morse cohomology groups one can use Conley index theory instead. If instead of $\R^n$ we take a separable Hilbert space and assume that the vector fields are compact perturbations of a fixed bounded self-adjoint Fredholm operator then the cohomological invariant, called $E$-cohomological Conley index, is still well defined (see \cite{Maciej},\cite{JDE} or \cite{StW}). The proof of Proposition \ref{prop:maxInvSet} remain unchanged in this case. Since the $E$-cohomolgocal Conley index is invariant under homotopies (Theorem 2.12 in \cite{JDE}) the conclusion of this note is also true for Hilbert spaces.

\subsection{Further questions}
We have proved that if two Morse-Smale functions have different Morse cohomology groups, then their gradients have to belong to different connected components of  $\mathcal{V}^\nabla$.  This leads to the question whether the map assigning a Morse-Smale function its Morse cohomology groups is injective on connected components of $\mathcal{V}^\nabla$. If the answer is affirmative, then one would like to know what are possible Morse cohomology groups for functions whose gradient is in $\mathcal{V}^\nabla$. 
\subsection*{Acknowledgements}
This research was partly supported by Grant Beethoven2 of the National Science Centre,
Poland, no. 2016/23/G/ST1/04081.

\end{document}